\documentclass[11pt]{amsart}

\usepackage[psamsfonts]{amssymb}
\usepackage{amsmath,amsfonts,amssymb,amscd}
\usepackage{graphicx}
\usepackage{latexsym,epsfig}
\usepackage[latin1]{inputenc}
\usepackage[T1]{fontenc}
\usepackage{enumerate}
\usepackage{amsbsy}
\usepackage{graphics}
\usepackage{pgf,tikz,verbatim}

\newtheorem{theorem}{Theorem}[section]
\newtheorem{corollary}[theorem]{Corollary}
\newtheorem{lemma}[theorem]{Lemma}
\newtheorem{prop}[theorem]{Proposition}

\newtheorem{example}[theorem]{Example}
\newtheorem{remark}[theorem]{Remark}

\newtheorem{definition}[theorem]{Definition}
\begin{document}

\title[Density of roots of the Yamada polynomial of spatial graphs]{Density of roots of the Yamada polynomial \\ of spatial graphs}

\author{Miaowang LI}
\address{School of Mathematical Sciences, Dalian University of Technology, Dalian 116024, P. R. China}
\email {limiaowang@yeah.net}

\author{Fengchun LEI}
\address{School of Mathematical Sciences, Dalian University of Technology, Dalian 116024, P. R. China}
\email{fclei@dlut.edu.cn}

\author{Fengling LI}
\address{School of Mathematical Sciences, Dalian University of Technology, Dalian 116024, P. R. China}
\email{dutlfl@163.com}

\author{Andrei VESNIN}
\address{Regional Scientific and Educational Mathematical Center, Department of Mathematics and Mechanics, Tomsk State University, Tomsk, 634050, Russia}
 \email{vesnin@math.nsc.ru}

\thanks{
The second author is supported in part by a grant (No.11431009) of NSFC; the third author is supported in part by grants (No.11671064 and No.11471151) of NSFC; the fourth author is supported by the Ministry of Education and Science of Russia (state assignment No.1.12877.2018/12.1).}

\subjclass[2010]{Primary 57M15; Secondary 05C31}

\keywords{Yamada polynomial, spatial graph, chain polynomial}

\begin{abstract}
We survey the construction and properties of the Yamada polynomial of spatial graphs and present the Yamada polynomial formulae for some classes of graphs. Then we construct an infinite family of spatial graphs for which roots of Yamada polynomials are dense in the complex plane.
\end{abstract}

\maketitle

\section{Introduction}

It is know that (Laurent) polynomials play an important role in the knot theory (see, for example, \cite{Kaw}). We just mention on Alexander polynomial and Jones polynomial which are seem known to any topologist. The spatial graph theory arises as a natural generalization of the knot theory. Recall that the study on intrinsic knotting and linking of graphs in $S^3$ was initiated the 1980s by the nice results of J.H.~Conway and C.~Gordon~\cite{CC}. They proved that any embedding of the complete graph $\mathbf K_7$ in $\mathbb R^3$ contains a knotted cycle and any embedding of the complete graph $\mathbf K_{6}$ in $\mathbb R^{3}$ contains a pair of linked cycles. It is natural that the modern theory of spatial graphs combines  topological and graph-theoretical methods. The powerful of polynomial invariants of knots as well as polynomial invariants of graphs was a natural motivation for investigation of polynomial invariants of spatial graphs initiated by L.H.~ Kauffman~\cite{k}. Being motivated by problems on knotting and linking of DNA and chemical compounds, the study of spatial graphs is in the center of interest for last decades.

It is well known due S.~Kinoshita~\cite{K1,K2}, that the Alexander ideal and Alexander polynomial are invariants of spatial graphs which are determined by the fundamental groups of the complements of spatial graphs.

In 1989, S.~Yamada~\cite{YA} introduced Yamada polynomial of spatial graphs in $\mathbb R^3$. It is an concise and useful ambient isotopy invariant for graphs with maximal degree less than four. There are many interesting results on Yamada polynomial and its generalizations. J. Murakami~\cite{M} investigated the two-variable extension $\mathbf{Z}_S$ of the Yamada polynomial and gave an invariant related to the HOMFLY polynomial. In 1994, the crossing number of spatial graphs in terms of the reduced degree of Yamada polynomial has been studied by T.~Motohashi, Y.~Ohyama and K.~Taniyama~\cite{MO}. In 1996, A.~Dobrynin and A.~Vesnin~\cite{DV} studied properties of the Yamada polynomial of spatial graphs. For any graph $G$, V.~Vershinin and A.~Vesnin~\cite{VV} defined bigraded cohomology groups whose graded Euler characteristic is a multiple of the Yamada polynomial of~$G$. Another invariant of spatial graphs associated with $U_{q}(sl(2, \mathbb C))$ was introduced by S.~Yoshinaga~\cite{YO}. See~\cite{Kobe} about the relation between Yamada polynomial and Yoshinaga polynomial.
Nice results on the structure of the Yamada and  flows polynomials of cubic graphs are established by I.~Agol and V.~Krushkal~\cite{AK}.

A polynomial invariant of  virtual graphs was constructed by Y.~Miyazawa~\cite{MY} as an extension of the Yamada polynomial in 2006, see also the paper~\cite{Fleming} by T.~ Fleming and B.~Mellor. The generalized Yamada polynomials of virtual spatial graphs were recently introduced by Q.~Deng, X.~Jin and L.H.~Kauffnan in~\cite{DJK}.

In this paper we will discuss zeros of Yamada polynomials. Recall that zeros of polynomial invariants of knots and graphs is a question of special interest studied by A.D.~Sokal~\cite{SO} and  P.~Csikv\'ari, P.E.~Frenkel, J.~Hladk\'y, T.~Hubai~\cite{CFHH}  for chromatic polynomial; by O.T.~Dashbach, T.D.~Le, X.-S.~Lin~\cite{DLL} and  X.~Jun, F.~Zhang, F.~Dong, E.G.~Tay~\cite{XF} for Jones polynomial.

This paper consists of the introduction and three parts. In the first part (Section~2) we recall properties of the Yamada polynomial of graphs and recall some formulae from~\cite{paper1} for computing the Yamada polynomial of graphs by edge replacements via the chain polynomial (see Theorem~\ref{tger}). In the second part (Section~3) we discuss formulae for computing the Yamada polynomial of spatial graphs obtained by replacing edges of  cycle graphs, theta-graphs, or bouquet graphs by spatial parts  (see Theorem~\ref{tcn} and~\cite{paper1}). In the last part (Section~4), we prove that zeros of the Yamada polynomial of spatial graphs are dense in the complex plane (see Theorem~\ref{tm4}).

\section{Yamada polynomial of a graph}

We consider a graph $G$, admitting loops and multiple edges. Let us use standard notations $p(G)$ and $q(G)$ for number of vertices and number of edges of it.

Before defining the Yamada polynomial of a graph, we recall a graph invariant which is a special case of the Negami polynomial invariant~\cite{N}.

\smallskip

\begin{definition}~\cite{YA}
Define 2-variable Laurent polynomial $h(G) = h(G) (x,y)$ of graph $G=(V,E)$, where $V = V(G)$ is the vertex set and $E = E(G)$ is the edge set of $G$, by the rule
$$
h(G)(x,y) = \sum_{F\subset E}(-x)^{-|F|} f(G-F),
$$
with $f(G)=x^{\mu(G)}y^{\beta(G)}$, where $\mu(G)$ and $\beta(G)$ is the number of connected components of $G$ and the first Betti number of $G$.
\end{definition}

\smallskip

\begin{definition}~\cite{YA}
{\it The Yamada polynomial} of a graph $G$ is a Laurent polynomial in one variable, obtained by the following substitution into $h(G)(x,y)$:
$$
H(G)(A) = h(G)(-1,-A-2-A^{-1}).
$$
\end{definition}

\smallskip

The following properties of $H(x,y)$ hold (see~\cite{YA} for details):
\begin{itemize}
\item[(1)] $H(\cdot)= -1$.
\item[(2)] Let $e$ be a non-loop edge of a graph $G$. Then $H(G)=H(G/e)+H(G-e)$, where $G/e$ is the graph obtained by contracting the edge $e$, and $G-e$ is the graph obtained by deleting of the edge $e$.
\item[(3)] Let $e$ be a loop edge of a graph $G$. Then $H(G)=-\sigma H(G-e)$, where $\sigma= A+1+A^{-1}$.
\item[(4)] Let $G_{1}\cup G_{2}$ be a disjoint union of graphs $G_{1}$and $G_{2}$, then $H(G_{1}\cup G_{2})= H(G_{1})H(G_{2})$.
\item[(5)] Let $G_{1}\cdot G_{2}$ be a union of graphs $G_{1}$ and $G_{2}$ having one common point, then
 \[H(G_{1}\cdot G_{2})= -H(G_{1})H(G_{2}).\]
\item[(6)] If $G$ has an isthmus, then  $H(G)= 0$.
\end{itemize}

It is easy to find directly (see also~\cite{DV}) polynomial $H(G)$ for some simple classes of graphs.

\smallskip

\begin{lemma}  \label{lemma-dv} The following properties holds with $\sigma = A + 1 + A^{-1}$.
\begin{itemize}
\item[(i)] Let $T_q$ be a tree with $q$ edges. Then $H(T_q)=0$ for all $q$.
\item[(ii)] Let $C_{n}$ be the cycle of length $n$. Then $H(C_{n})= \sigma$ for all $n$.
\item[(iii)] Let $B_q$ be the one-vertex graph with $q$ loops, also known as ``$q$-bouquet''. Then $H(B_q) \, = \,(-1)^{q-1} \sigma^q$.
\item[(iv)] Let $\Theta_s$ be the graph consisting of two vertices and $s$ edges between them, also known as ``$s$-theta-graph'' (see Figure~\ref{fig1}). Then
$$
 H(\Theta_s) \, = \, \frac{1}{\sigma+1} [\sigma +(-\sigma)^s].
$$
\end{itemize}
\end{lemma}

\begin{figure}[!ht]
\centering
\unitlength=0.6mm
\begin{picture}(0,35)(0,0)
\thicklines
\put(-20,15){\circle*{3}}
\put(20,15){\circle*{3}}
\qbezier(-20,15)(0,40)(20,15)
\qbezier(-20,15)(0,25)(20,15)
\put(0,33){\makebox(0,0)[cc]{\small $1$}}
\put(0,23){\makebox(0,0)[cc]{\small $2$}}
\put(0,5){\makebox(0,0)[cc]{\small $s$}}
\put(0,15){\makebox(0,0)[cc]{$\vdots$}}
\qbezier(-20,15)(0,-10)(20,15)
\end{picture}
\caption{The graph $\Theta_{s}$.} \label{fig1}
\end{figure}
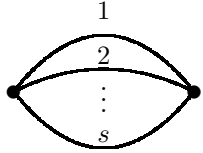

\smallskip

The following property of $H(G)$ was obtained in \cite{paper1}.

\smallskip

\begin{prop}\label{ptv} \cite{paper1}
Let $G_{1}:G_{2}$ be the union of two graphs $G_{1}$ and $G_{2}$ having only two common vertices $u$ and $v$. Let $K_{1}$ and $K_{2}$ be graphs obtained from $G_{1}$ and $G_{2}$, respectively, by identifying $u$ and $v$. Then
\begin{equation*}\label{twovert}
H(G_{1}:G_{2})  =  \frac{1}{\sigma} \Big[ H(K_{1}) H(K_{2}) + (\sigma+1) H(G_{1}) H(G_{2})   + H(K_{1}) H(G_{2}) + H(K_{2}) H(G_{1}) \Big].
\end{equation*}
\end{prop}

\smallskip

Before further discussions we recall some properties of the chain polynomial introduced by R.C.~Read and E.G.~Whitehead~\cite{RE}, see also~\cite{XF}.  The chain polynomial is defined for edge-labelled graphs where labels are elements of a commutative ring with unity. We will denote the edges by the labels associated with them.

\smallskip

\begin{definition}
{\it The chain polynomial} $\operatorname{Ch}(G)$ of a labelled graph $G$ is defined as
$$
\operatorname{Ch}(G) = \sum_{Y \subset E} F_{G-Y} (1-w) \prod_{a \in Y} a,
$$
where the sum is taken over all subsets of the edge set $E$ of $G$, $F_{G-Y} (1-w)$ denotes the flow polynomial of the subgraph $G - Y$ calculated at $1-w$, and $\prod_{a\in Y}$ denotes the product of edge-labels of $Y$.
\end{definition}

\smallskip

The chain polynomial can be also defined in the following recursive form.

\smallskip

\begin{definition} \label{def3.2}
{\it The chain polynomial}  $\operatorname{Ch} (G) (w)$ in a variable $w$ of a labelled graph $G$ is defined by following rules.
\begin{itemize}
\item[(1)] If $G$ is edgeless, then $\operatorname{Ch}(G)=1$.
\item[(2)] Otherwise, suppose $a$ is an edge of $G$ labelled by $a$. Then
\begin{itemize}
\item[(2a)] If $a$ is a loop, then $\operatorname{Ch}(G) = (a-w) \operatorname{Ch}(G-a)$.
\item[(2b)] If $a$ is not a loop, then $\operatorname{Ch}(G) = (a-1) \operatorname{Ch} (G-a) + \operatorname{Ch}(G/a)$.
\end{itemize}
\end{itemize}
\end{definition}

\smallskip

For a reader convenience we demonstrate chain polynomials for simplest classes of graphs.

\smallskip

\begin{example}\label{Ecn}
Let $C_{n}$ be the $n$-cycle with edges labelled by $a_{1}, a_{2}, \cdots, a_{n}$, then
$$
\operatorname{Ch}(C_{n}) = \prod_{i=1}^n a_i \, - \, w.
$$
Let $\Theta_{s}$ be the $s$-theta-graph with edges labelled by $a_{1}, a_{2}, \cdots, a_{s}$, then
$$
\operatorname{Ch} (\Theta_{s}) = \frac{1}{1-w} \Big[ \prod^{s}_{i=1}(a_{i}-w) - w \prod ^s_{i=1}(a_{i}-1) \Big].
$$
Let $B_{q}$ be the $q$-bouquet with $q$ loops labelled by  $a_{1}, a_{2}, \ldots,  a_{q}$, then
$$
\operatorname{Ch} ( B_{q}) = \prod^{q}_{i=1}(a_{i}-w).
$$
\end{example}

\smallskip

To explore the relation between the chain polynomial and the Yamada polynomial, inspired by \cite{XF}, let us introduce the following notation. Let $G$ be a connected labelled graph and $K_{E}$ be a family of connected graphs with two attached vertices. Denote by $G(K_{E})$ the graph obtained from $G$ by replacing each edge $a= uv$ of $G$ by a connected graph $K_{a}\in K_{E}$ with two attached vertices $u$ and $v$ that has only the vertices $u$ and $v$ in common with $(G-a)(K_{E})$.

Let $K'_{a}$ be the graph obtained from $K_{a}$ by identifying $u$ and $v$, the two attached
vertices of $K_{a}$. Denote
$$
\alpha_{a }= \alpha(K_{a}) := \frac{1}{\sigma}[(\sigma+1)H(K_{a})+H(K'_{a})],
$$
$$
\beta_{a} = \beta(K_{a}) := \frac{1}{\sigma}[H(K_{a})+H(K'_{a})],
$$
and
$$
\gamma_{a} = \gamma(K_{a}) : =1-\frac{\alpha (K_{a})}{\beta (K_{a})}.
$$
It is easy to see that
$$
H(K'_{a})=(\sigma+1)\beta_{a}-\alpha_{a} \qquad \text{and} \qquad  H(K_{a})=\alpha_{a}-\beta_{a}.
$$

The following result from \cite{paper1} gives the relation between Yamada polynomial and chain polynomial.

\smallskip

\begin{theorem}\label{tger} \cite{paper1}
Let $G$ be a connected labelled graph, and $G(K_{E})$ be the graph obtained from $G$ by replacing the edge $a$ by a connected graph $K_{a}\in K_{E}$ for every edge $a$ of $G$. If we replace $w$ by $-\sigma$, and replace $a$ by $\gamma_{a}$ for every label $a$ in $\operatorname{Ch}(G)$, then we get
\begin{equation*}
H(G(K_{E}))=\frac{\prod_{a\in E(G)}\beta_{a}}{(-1)^{q(G)-p(G)}} \operatorname{Ch}(G),
\end{equation*}
where $p(G)$ and $q(G)$ are the numbers of vertices and edges of $G$, respectively.
\end{theorem}

\smallskip

\begin{example}
Consider a cyclic graph $G = C_{n}$ with all edges labelled by $a$ and replace each edge of it by the graph $K_{a} = \Theta_{s}$. These graphs as well as the resulting graph $C_{n}(\Theta_{s})$ are presented in Figure~\ref{fig2} for $n=4$ and $s=3$.
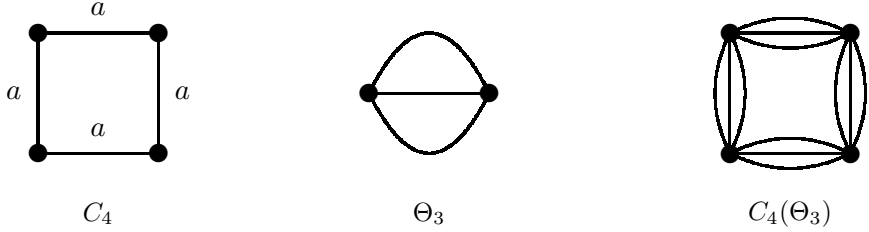
\begin{figure}[!ht]
\centering
\unitlength=0.8mm
\begin{picture}(0,40)
\put(-60,0){
\begin{picture}(0,30)(-5,10)
\thicklines
\put(-10,20){\circle*{3}}
\put(10,20){\circle*{3}}
\put(-10,40){\circle*{3}}
\put(10,40){\circle*{3}}
\qbezier(-10,20)(-10,20)(-10,40)
\qbezier(10,20)(10,20)(10,40)
\qbezier(-10,20)(-10,20)(10,20)
\qbezier(-10,40)(-10,40)(10,40)
\put(-14,30){\makebox(0,0)[cc]{$a$}}
\put(14,30){\makebox(0,0)[cc]{$a$}}
\put(0,44){\makebox(0,0)[cc]{$a$}}
\put(0,24){\makebox(0,0)[cc]{$a$}}
\put(0,10){\makebox(0,0)[cc]{\small $C_{4}$}}
\end{picture}}
\put(0,10){
\begin{picture}(0,40)(0,10)
\thicklines
\put(-10,20){\circle*{3}}
\put(10,20){\circle*{3}}
\qbezier(-10,20)(0,20)(10,20)
\qbezier(-10,20)(0,00)(10,20)
\qbezier(-10,20)(0,40)(10,20)
\put(0,0){\makebox(0,0)[cc]{\small $\Theta_{3}$}}
\end{picture}}
\put(60,0){
\begin{picture}(0,60)(0,10)
\thicklines
\put(-10,20){\circle*{3}}
\put(10,20){\circle*{3}}
\qbezier(-10,20)(0,15)(10,20)
\qbezier(-10,20)(0,20)(10,20)
\qbezier(-10,20)(0,25)(10,20)
\put(-10,40){\circle*{3}}
\put(10,40){\circle*{3}}
\qbezier(-10,40)(0,35)(10,40)
\qbezier(-10,40)(0,40)(10,40)
\qbezier(-10,40)(0,45)(10,40)
\qbezier(-10,20)(-15,30)(-10,40)
\qbezier(-10,20)(-10,30)(-10,40)
\qbezier(-10,20)(-5,30)(-10,40)
\qbezier(10,20)(15,30)(10,40)
\qbezier(10,20)(10,30)(10,40)
\qbezier(10,20)(5,30)(10,40)
\put(0,10){\makebox(0,0)[cc]{\small $C_{4} (\Theta_{3})$}}
\end{picture}}
\end{picture}
\caption{Graphs $G = C_{n}$, $K_{a} = \Theta_{s}$ and $C_{n} (\Theta_{s})$ for $n=4$ and $s=3$.} \label{fig2}
\end{figure}

Since in this case $K_{a}^{'} = B_{s}$, we get:
$$
H(K_{a}) := H(\Theta_{s}) = \frac{1}{\sigma + 1} \left[ \sigma + (-\sigma)^{s} \right], \qquad  H(K'_{a}) := H(B_{s}) = (-1)^{s-1} \sigma^{s}.
$$
Denote
$$
\alpha := \alpha_{a}  = \frac{1}{\sigma} \left[  (\sigma + 1) H(\Theta_{s}) + H(B_{s}) \right], \quad
\beta := \beta_{a} = \frac{1}{\sigma} \left[ H(\Theta_{s}) + H(B_{s}) \right],  \quad \text{\rm and} \quad  \gamma := \gamma_{a}.
$$
Recall that $\operatorname{Ch} (C_{n}) = \prod_{i=1}^{n} a_{i} - w$ and $p(C_{n}) = q(C_{n}) = n$.
Then
\begin{eqnarray*}
H(C_{n} (\Theta_{s})) &  = &  \beta^{n} ( \gamma^{n} - (-\sigma) ) = (\beta \gamma)^{n} + \sigma \beta^{n} = (\beta - \alpha)^{n} + \sigma \beta^{n} = ( - H(\Theta_{s}))^{n} + \sigma \beta^{n} \cr
& = &  \left( - \frac{1}{\sigma + 1} \left[ \sigma + (-\sigma)^{s} \right] \right)^{n} + \sigma \frac{1}{\sigma^{n}} \left[ H(\Theta_{s}) + H(B_{s}) \right]^{n} \cr
& = & \left[ - \frac{\sigma + (-\sigma)^{s}}{\sigma + 1} \right]^{n} + \frac{1}{\sigma^{n-1}} \left[ \frac{\sigma + (-\sigma)^{s}}{\sigma + 1} + (-1)^{s-1} \sigma^{s} \right]^{n}.
\end{eqnarray*}
\end{example}

\section{Yamada polynomial of spatial graphs}

Next we will consider the Yamada polynomial of spatial graphs. Through the paper we work in the piecewise-linear category.  Let $G$ be a graph embedded in $\mathbb R^3$, we say $G$ is a \emph{spatial graph}. If for each vertex $v$ of $G$ there exists a neighborhood $U_{v}$  of $v$ and a flat plane $P_{v}$ such that $G \cap U_{v} \subset P_{v}$, then we say that $G$ is a \emph{flat vertex graph}. For two spatial graphs $G$ and $G'$, if there exists an isotopy $\phi_{t} : \mathbb R^{3} \to \mathbb R^{3}$, $t \in [0,1]$ such that $\phi_{0} = \operatorname{id}$ and $\phi_{1} (G) = G'$, then we say that $G$ and $G'$ are \emph{ambient isotopic as  pliable vertex graphs (pliably isotopic)}. For two flat vertex graphs $G$, $G'$, if there exists an isotopy $\phi_{t} : \mathbb R^{3} \to \mathbb R^{3}$ with $t \in [0,1]$ such that $h_{0} = \operatorname{id}$, $h_{1}(G) = (G')$,  and $h_{t}(G)$ are flat vertex graphs for each $t \in [0,1]$, then we say that $G$ and $G'$ are \emph{ambient isotopic as flat vertex graphs (flatly isotopic)}.

Let $G \subset \mathbb R^{3}$ be a spatial graph. According to~\cite{YA}, we say that a projection $p : \mathbb R^{3} \to \mathbb R^{2}$ is regular projection corresponding to $G$ if each multi point of $p(G)$ is a double point with two transversal edges. The projection with information about the over/under crossing at all crossings is a \emph{diagram} of $G$.  Classical Reidemeister moves $\mathcal R_{0}$, $\mathcal R_{1}$, $\mathcal R_{2}$ and $\mathcal R_{3}$  of knot diagrams and  Reidemeister moves $\mathcal R_{4}$, $\mathcal R_{5}$ and $\mathcal R_{6}$ of neighbourhoods of graphs vertices are presented in Figure~\ref{fig3}.
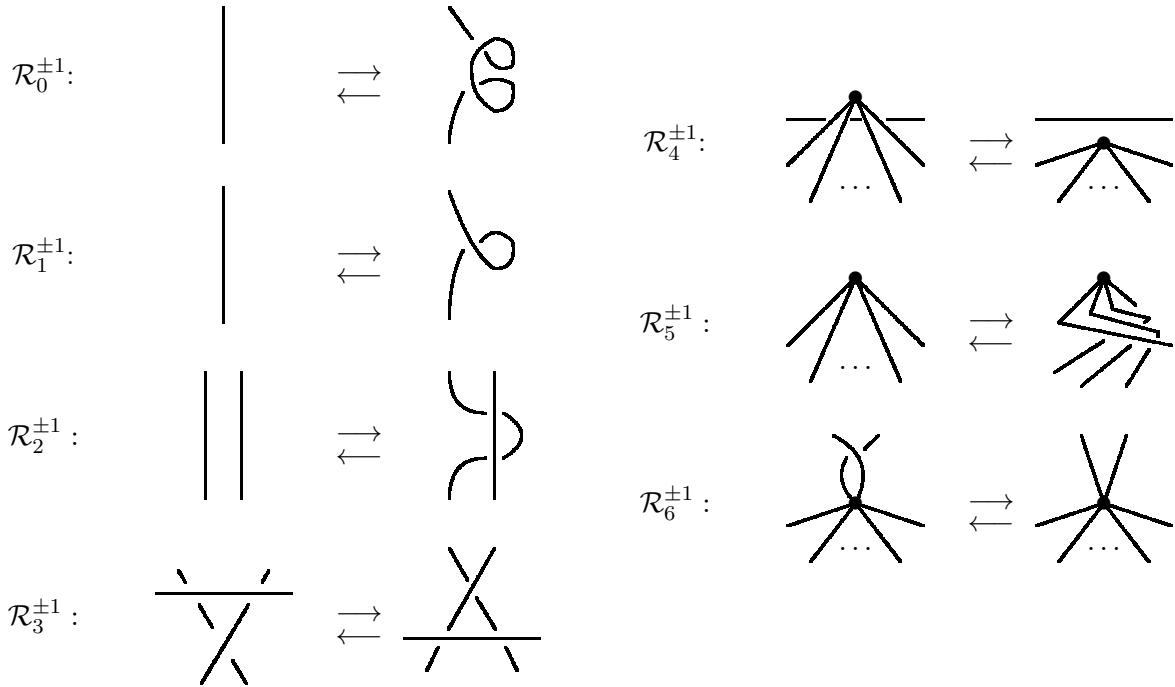
\begin{figure}[h]
\begin{center}
\unitlength=.6mm
\begin{picture}(0,150)(0,30)
\thicklines
\put(-70,0){\begin{picture}(0,30)
 \put(-50,160){\makebox(0,0)[cc]{$\mathcal R_0^{\pm 1}$:}}
 \qbezier(-10,145)(-10,145)(-10,175)
 \put(20,160){\makebox(0,0)[cc]{$\longrightarrow$}}
\put(20,155){\makebox(0,0)[cc]{$\longleftarrow$}}
\qbezier(40,145)(40,150)(43,156)  \qbezier(54,158)(50,160)(47,158) \qbezier(50,152)(55,152)(54,158)
\qbezier(45,160)(45,155)(50,152) \qbezier(45,160)(44,165)(50,168) \qbezier(50,168)(55,168)(54,162)
\qbezier(54,162)(50,160)(48,165)
\qbezier(40,175)(40,175)(45,168)
 \put(-50,120){\makebox(0,0)[cc]{$\mathcal R_1^{\pm 1}$:}}
\qbezier(-10,105)(-10,105)(-10,135)
\put(20,120){\makebox(0,0)[cc]{$\longrightarrow$}}
\put(20,115){\makebox(0,0)[cc]{$\longleftarrow$}}
\qbezier(40,106)(40,115)(43,121) \qbezier(40,134)(45,120)(50,117)
\qbezier(50,117)(55,117)(54,123) \qbezier(54,123)(50,127)(47,123)
\put(-50,80){\makebox(0,0)[cc]{$\mathcal R_2^{\pm 1}:$}}
\qbezier(-14,66)(-14,66)(-14,94) \qbezier(-6,66)(-6,66)(-6,94)
\put(20,80){\makebox(0,0)[cc]{$\longrightarrow$}}
\put(20,75){\makebox(0,0)[cc]{$\longleftarrow$}}
\qbezier(50,66)(50,66)(50,94) \qbezier(40,94)(40,85)(48,85)
\qbezier(40,66)(40,75)(48,75) \qbezier(52,75)(60,80)(52,85)
\put(-50,40){\makebox(0,0)[cc]{$\mathcal R_3^{\pm 1}:$}}
\put(20,40){\makebox(0,0)[cc]{$\longrightarrow$}}
\put(20,35){\makebox(0,0)[cc]{$\longleftarrow$}}
\qbezier(-25,45)(-25,45)(5,45) \qbezier(-1.5,47.5)(-1.5,47.5)(0,50)
\qbezier(-15,25)(-15,25)(-4.5,42.5)
\qbezier(-20,50)(-20,50)(-18.5,47.5)
\qbezier(-15.5,42,5)(-15.5,42.5)(-12.5,37.5)
\qbezier(-8,30)(-8,30)(-5,25)
\qbezier(30,35)(30,35)(60,35)
\qbezier(35,28)(35,28)(37.5,33)
\qbezier(39.75,37)(39.75,37)(50,55)
\qbezier(55,28)(55,28)(52.5,33)
\qbezier(50.25,37)(50.25,37)(46,44)
\qbezier(40,55)(40,55)(44,48)
\end{picture}}
\put(70,25){\begin{picture}(0,0)
\thicklines \put(-50,120){\makebox(0,0)[cc]{$\mathcal R_4^{\pm 1}$:}}
\put(-10,130){\circle*{3}}
\qbezier(-25,115)(-25,115)(-10,130)
\qbezier(-20,107)(-20,107)(-10,130)
\qbezier(0,107)(0,107)(-10,130)
\qbezier(5,115)(5,115)(-10,130)
\put(-10,110){\makebox(0,0)[cc]{$\ldots$}}
\qbezier(-25,125)(-25,125)(-17,125)
\qbezier(-11,125)(-11,125)(-9,125)
\qbezier(-3,125)(-3,125)(5,125)
\put(20,120){\makebox(0,0)[cc]{$\longrightarrow$}}
\put(20,115){\makebox(0,0)[cc]{$\longleftarrow$}}
\qbezier(30,115)(30,115)(45,120)
\qbezier(35,107)(35,107)(45,120)
\qbezier(55,107)(55,107)(45,120)
\qbezier(60,115)(60,115)(45,120)
\put(45,110){\makebox(0,0)[cc]{$\ldots$}}
\qbezier(30,125)(30,125)(60,125)
\put(45,120){\circle*{3}}
\put(-50,80){\makebox(0,0)[cc]{$\mathcal R_5^{\pm 1}:$}}
\put(-10,90){\circle*{3}}
\qbezier(-25,75)(-25,75)(-10,90)
\qbezier(-20,67)(-20,67)(-10,90)
\qbezier(0,67)(0,67)(-10,90)
\qbezier(5,75)(5,75)(-10,90)
\put(-10,70){\makebox(0,0)[cc]{$\ldots$}}
\put(20,80){\makebox(0,0)[cc]{$\longrightarrow$}}
\put(20,75){\makebox(0,0)[cc]{$\longleftarrow$}}
\put(45,90){\circle*{3}}
\qbezier(35,80)(35,80)(45,90) \qbezier(35,80)(35,80)(60,75)
\qbezier(42,82)(42,82)(45,90) \qbezier(42,82)(42,82)(57,78) \qbezier (57,78)(57,78)(57,77) \qbezier(55,74)(55,74)(50,66)
\qbezier(47,83)(47,83)(45,90) \qbezier(47,83)(47,83)(55,81) \qbezier(55,81)(55,81)(54,80) \qbezier(51,75)(51,75)(40,66)
\qbezier(52,84)(52,84)(45,90) \qbezier(45,76)(45,76)(34,69)
\put(-50,40){\makebox(0,0)[cc]{$\mathcal R_6^{\pm 1}:$}}
\put(20,40){\makebox(0,0)[cc]{$\longrightarrow$}}
\put(20,35){\makebox(0,0)[cc]{$\longleftarrow$}}
\qbezier(-25,35)(-25,35)(-10,40)
\qbezier(-20,27)(-20,27)(-10,40)
\qbezier(0,27)(0,27)(-10,40)
\qbezier(5,35)(5,35)(-10,40)
\put(-10,30){\makebox(0,0)[cc]{$\ldots$}}
\qbezier(-10,40)(-5,50)(-15,55)
\qbezier(-10,40)(-15,45)(-12,50)
\qbezier(-8,52)(-8,52)(-5,55)
\put(-10,40){\circle*{3}}
\qbezier(30,35)(30,35)(45,40)
\qbezier(35,27)(35,27)(45,40)
\qbezier(55,27)(55,27)(45,40)
\qbezier(60,35)(60,35)(45,40)
\put(45,30){\makebox(0,0)[cc]{$\ldots$}}
\qbezier(45,40)(45,40)(40,55)
\qbezier(45,40)(45,40)(50,55)
\put(45,40){\circle*{3}}
\end{picture} }
\end{picture}
\end{center} \caption{The Reidemeister moves.} \label{fig3}
\end{figure}
We say the deformation generated by $\mathcal R_{1}, \ldots, \mathcal R_{6}$ is \emph{pliable deformation}, one generated by $\mathcal R_{1}, \ldots, \mathcal R_{5}$ is a \emph{flat deformation}, ant we say one generated by $\mathcal R_{0}$ and $\mathcal R_{2}, \ldots, \mathcal R_{4}$ is a \emph{regular deformation}.

Let $g$ be a diagram of $G$. For any double point, S.~Yamada~\cite{YA} defined the spins of $+1$, $-1$ and $0$, which are denoted by $S_+$, $S_-$ and $S_0$, as shown in Figure~\ref{fig4}.
\begin{figure}[!ht]
\centering
\unitlength=0.8mm
\begin{picture}(0,40)(0,0)
\thicklines
\qbezier(-80,10)(-80,10)(-60,30)
\qbezier(-80,30)(-80,30)(-72,22)
\qbezier(-60,10)(-60,10)(-68,18)
\put(-40,20){\makebox(0,0)[cc]{$\longrightarrow$}}
\qbezier(-20,10)(-10,20)(-20,30)
\qbezier(0,10)(-10,20)(0,30)
\put(-10,0){\makebox(0,0)[cc]{$S_{+}$}}
\put(30,0){\makebox(0,0)[cc]{$S_{-}$}}
\put(70,0){\makebox(0,0)[cc]{$S_{0}$}}
\qbezier(20,30)(30,20)(40,30)
\qbezier(20,10)(30,20)(40,10)
\qbezier(60,30)(80,10)(80,10)
\qbezier(60,10)(80,30)(80,30)
\put(70,20){\circle*{3}}
\end{picture}
\caption{Spins of $+1$, $-1$, and $0$.}  \label{fig4}
\end{figure}
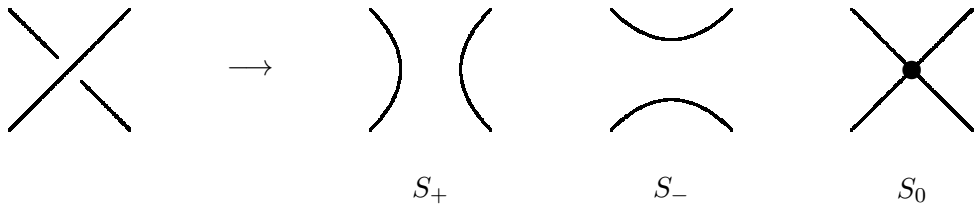

Let $S$ be the plane graph obtained from $g$ by replacing each double point with a spin.  $S$ is called {\it a state} on $g$. Denote the set of all states by $U(g)$. Put
$$
c(g|S)= A^{m_1-m_2},
$$
where $m_1$ and $m_2$ are the numbers of double points with spin $S_+$ and $S_-$, respectively, used to obtain $S$ from $g$.

\smallskip

\begin{definition}~\cite{YA}\label{dsg}
{\it The Yamada polynomial} of a diagram $g$ of a spatial graph $G$ is a Laurent polynomial in variable $A$ defined as follows
$$
R[g] = R[g](A) : =\sum_{S\in U(g)} c(g|S)H(S),
$$
where $H(S) = h(S)(-1,-A-2-A^{-1})$.
\end{definition}

The role of Yamada polynomial in studying of spatial graphs is clear from the following result.

\smallskip

\begin{theorem} \cite{YA} Yamada polynomial has the following properties.
\begin{enumerate}
\item $R[g]$ is a regular deformation invariant of a  diagram $g$.
\item $R[g]$ is a flat deformation invariant of a diagram $g$ up to multiplying $(-A)^{n}$ for some integer $n$.
\item If $g$ is a diagram of a graph whose maximum degree is less then $4$, then $R(g)$ is a pliable deformation invariant of $g$ up to multiplying $(-A)^{n}$ for some integer $n$.
\end{enumerate}
\end{theorem}

\smallskip

Here we recall only some properties of the Yamada polynomial.

\smallskip

\begin{prop}~\cite{YA}\label{pr}
The following properties hold.
\begin{enumerate}
\item Let $g_{1}\cup g_{2}$ be a disjoint union of diagrams $g_{1}$ and $g_{2}$, then
$$R[g_{1}\cup g_{2}]= R[g_{1}]R[g_{2}].$$
\item  Let $g_{1}\cdot g_{2}$ be a union of diagrams $g_{1}$ and $g_{2}$ having one common point, then
$$
R[g_{1}\cdot g_{2}] = -R[g_{1}]R[g_{2}].
$$
\item If $g$ has an isthmus, then  $R[g]= 0$.
\end{enumerate}
\end{prop}

\smallskip

\begin{remark}
If a diagram $g$ of $G$ does not have double points, then $R[g]= H(G)$.
\end{remark}

Properties of $H(G)$ imply immediately the following properties of $R(G)$.

\smallskip

\begin{prop} \cite{YA}  \label{pv} The following properties hold.

\unitlength=0.6mm
\begin{center}
\begin{picture}(120,14)(91,4)
\thicklines
\put(30,10){\makebox(0,0)[c]{(1)}}
\put(46,10){\makebox(0,0)[c]{$R \, [$}}
\qbezier(52,5)(52,5)(62,15)
\qbezier(52,15)(52,15)(55,12)
\qbezier(59,8)(59,8)(62,5)
\put(65,10){\makebox(0,0)[c]{$]$}}
\put(75,10){\makebox(0,0)[c]{$=$}}
\put(90,10){\makebox(0,0)[c]{$A \; R \, [$}}
\put(100,10){\oval(6,10)[r]}
\put(110,10){\oval(6,10)[l]}
\put(113,10){\makebox(0,0)[c]{$]$}}
\put(120,10){\makebox(0,0)[c]{$+$}}
\put(136,10){\makebox(0,0)[c]{$A^{-1} R \, [$}}
\put(152,15){\oval(10,6)[b]}
\put(152,5){\oval(10,6)[t]}
\put(160,10){\makebox(0,0)[c]{$]$}}
\put(167,10){\makebox(0,0)[c]{$+$}}
\put(176,10){\makebox(0,0)[c]{$R \, [$}}
\qbezier(182,15)(182,15)(192,5)
\qbezier(182,5)(182,5)(192,15)
\put(187,10){\circle*{3}}
\put(195,10){\makebox(0,0)[c]{$] \, ;$}}
\end{picture}

\begin{picture}(120,10)(91,9)
\thicklines
\put(30,10){\makebox(0,0)[c]{(2)}}
\put(46,10){\makebox(0,0)[c]{$R \, [$}}
\qbezier(52,5)(52,5)(57,10)
\qbezier(57,10)(57,10)(52,15)
\qbezier(57,10)(57,10)(52,12)
\put(52.5,10){\circle*{1}}
\put(52.5,8){\circle*{1}}
\qbezier(57,10)(57,10)(67,10)
\qbezier(72,5)(72,5)(67,10)
\qbezier(72,12)(72,12)(67,10)
\qbezier(72,15)(72,15)(67,10)
\put(71.5,10){\circle*{1}}
\put(71.5,8){\circle*{1}}
\put(74.5,10){\makebox(0,0)[c]{$]$}}
\put(83.5,10){\makebox(0,0)[c]{$=$}}
\put(94,10){\makebox(0,0)[c]{$R \, [$}}
\qbezier(100,5)(100,5)(105,10)
\qbezier(105,10)(105,10)(100,15)
\qbezier(105,10)(105,10)(100,12)
\put(100.5,10){\circle*{1}}
\put(100.5,8){\circle*{1}}
\qbezier(115,10)(115,10)(120,5)
\qbezier(115,10)(115,10)(120,12)
\qbezier(115,10)(115,10)(120,15)
\put(119.5,10){\circle*{1}}
\put(119.5,8){\circle*{1}}
\put(122.5,10){\makebox(0,0)[c]{$]$}}
\put(131,10){\makebox(0,0)[c]{$+$}}
\put(141,10){\makebox(0,0)[c]{$R \, [$}}
\qbezier(147,5)(147,5)(152,10)
\qbezier(152,10)(152,10)(147,15)
\qbezier(152,10)(152,10)(147,12)
\put(147.5,10){\circle*{1}}
\put(147.5,8){\circle*{1}}
\qbezier(152,10)(152,10)(157,5)
\qbezier(152,10)(152,10)(157,12)
\qbezier(152,10)(152,10)(157,15)
\put(156.5,10){\circle*{1}}
\put(156.5,8){\circle*{1}}
\put(160.5,10){\makebox(0,0)[c]{$] \, ;$}}
\end{picture}

\begin{picture}(120,14)(91,8)
\thicklines
\put(30,10){\makebox(0,0)[c]{(3)}}
\put(46,10){\makebox(0,0)[c]{$R \, [$}}
\put(57,10){\circle{8}}
\put(53,10){\circle*{3}}
\put(65,10){\makebox(0,0)[c]{$]$}}
\put(72,10){\makebox(0,0)[c]{$=$}}
\put(82,10){\makebox(0,0)[c]{$R \, [$}}
\put(93,10){\circle{8}}
\put(101,10){\makebox(0,0)[c]{$]$}}
\put(110,10){\makebox(0,0)[c]{$=$}}
\put(136,10){\makebox(0,0)[c]{$A + 1 + A^{-1} \, ;$}}
\end{picture}
\end{center}
\smallskip
\smallskip

(4) $R[B_n]=-(-\sigma)^n$, where $B_n$ is the unlinked one-vertex graph with $n$ loops, and $\sigma = A+1+A^{-1}$.

\end{prop}

\smallskip

The formula from Proposition \ref{ptv} about $H(G)$ implies the similar  formula for polynomial $R[g]$.

\smallskip

\begin{prop}\label{cdtv} \cite{paper1}
Let  $g_{1}:g_{2}$ be the union of two diagrams $g_{1}$ and $g_{2}$ having only two common vertices $u$ and $v$. Let $k_{1}$ and $k_{2}$ be diagrams obtained from $g_{1}$ and $g_{2}$, respectively, by identifying $u$ and $v$. Then
\begin{equation*} 
R[g_{1}:g_{2}] = \frac{1}{\sigma} \Big( R[k_{1}] R[k_{2}] + (\sigma+1) R[g_{1}] R[g_{2}]  + R[k_{1}] R[g_{2}] + R[k_{2}] R[g_{1}] \Big).
\end{equation*}
\end{prop}

\smallskip

Let $G$ be a connected plane labelled graph, where each edge $a$ with terminal vertices $u$ and $v$ is labelled by a diagram $g_{a}$ having  two vertices $u_{a*}$ and $v_{a*}$ indicated. We define $G(g_{a})$ to be the spatial graph (its diagram) obtained from $G$ by replacing an edge $a= u_{a}v_{a}$ of $G$ by a connected diagram $g_{a}$ applying an  identification $u_{a}$ with $u_{a*}$ and an identification $v_{a}$ with $v_{a*}$ in such a way that  $g_{a}$ has only vertices $u_{a}=u_{a*}$ and $v_{a}=v_{a*}$ in common with $G-a$. We define $G(g_{a}, g_{b})$ to be the spatial graph (its diagram) obtained from $G(g_{a})$ by replacing an edge $b= u_{b}v_{b}$ of $G(g_{a})\cap (G-a)$  by a connected diagram $g_{b}$ applying identifications $u_{b}$ with $u_{b*}$ and $v_{b}$ with $v_{b*}$ in such a way that  $g_{b}$ has only vertices $u_{b}=u_{b*}$ and $v_{b}=v_{b*}$ in common with $G(g_{a})$. With the same construction, we can obtain $G(g_{a_1}, g_{a_2}, g_{a_3}, \ldots)$.
In this context we denote by $g_{a}'$ the diagram obtained from $g_{a}$ by identifying vertices $u_{a*}$ and $v_{a*}$ in such a way that no new intersections appear.

\smallskip

\begin{theorem}\label{tcn} \cite{paper1} The following properties hold.
\begin{itemize}
\item[(1)] Let $C_{n}$ be the $n$-cycle graph with edges labelled by $a_{1}, a_{2}, \ldots, a_{n}$, then
\begin{equation}\label{ecn}
R[C_{n}(g_{a_1}, g_{a_2}, \cdots, g_{a_n})] = \prod_{i=1}^n(-R[g_{a_i}])+\sigma\prod_{i=1}^n \Big( \frac{R[g_{a_i}]+R[g_{a_i}']}{\sigma} \Big).
\end{equation}
\item[(2)] Let $\Theta_{s}$ be the $s$-theta-graph with edges labelled by $a_{1}, a_{2}, \ldots, a_{s}$, then
\begin{equation*}\label{eths}
R[\Theta_{s}(g_{a_1}, g_{a_2}, \cdots, g_{a_s})] 
= \frac{(-1)^{s}}{1+\sigma} \prod_{i=1}^sR[g_{a_i}'] + \frac{\sigma}{1+\sigma} \prod_{i=1}^s \Big(\frac{\sigma+1}{\sigma}R[g_{a_i}]+\frac{1}{\sigma}R[g_{a_i}'] \Big).
\end{equation*}
\item[(3)] Let $B_{q}$ be the $q$-bouquet with loops labelled by $a_{1}, a_{2}, \ldots, a_{q}$, then
\begin{equation*}\label{ebp}
R[B_{q}(g_{a_1}, g_{a_2}, \cdots, g_{a_q})] = (-1)^{q-1}\prod_{i=1}^q R[g_{a_i}'].
\end{equation*}
\end{itemize}
\end{theorem}

\smallskip

If every edge of $G$ is labelled by the same diagram $g_{a} = g$, we denote $G(g, \ldots, g)$ shortly by $G(g)$.

\smallskip

\begin{corollary}\label{ccn} \cite{paper1}
If all edges of a graph $G$ are labelled by the same diagram $g$, then
\begin{equation*}
R[C_{n}(g)]=(-R[g])^{n}+\sigma \Big( \frac{R[g]+R[g']}{\sigma} \Big)^{n},
\end{equation*}
\begin{equation*}
R[\Theta_{s}(g)] = \frac{(-1)^{s}}{1+\sigma}R[g']^{s} + \frac{\sigma}{1+\sigma} \Big( \frac{\sigma+1}{\sigma}R[g]+\frac{1}{\sigma}R[g'] \Big)^{s},
\end{equation*}
and
\begin{equation*}
R[B_{q}(g)] = (-1)^{q-1} R[g']^q.
\end{equation*}
\end{corollary}

Remark that  $C_{n}(g_{a_1}, g_{a_2}, \ldots, g_{a_n})$ is the ``ring of beads'' discussed in~\cite{xj} and~\cite{RE}. By Eq.(\ref{ecn}), we get that Yamada polynomial $R[C_{n}(g_{a_1}, g_{a_2}, \ldots, g_{a_n})]$ of a  graph $C_{n}(g_{a_1}, g_{a_2}, \ldots, g_{a_n})$ is independent of the order in which the subdiagrams occur.

The following example illustrates the applying of Theorem \ref{tcn}.

\begin{example}\label{E4}
Let $\infty_{+}$ be the spatial graph diagram with two vertices and one double point signed by ``$+$'', see Figures~\ref{fig5} and~\ref{fig6}. We denote the mirror image of diagram $\infty_{+}$ by $\infty_{-}$. Direct calculations give Yamada polynomials $R[\infty_{+}]=A^{-2}\sigma$ and $R[\infty_{+}']=\sigma$, where $\sigma = A + 1 + A^{-1}$.

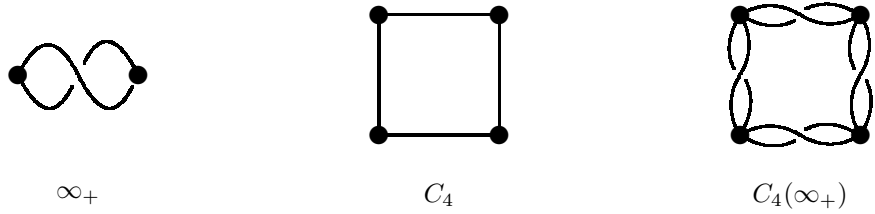
\begin{figure}[!ht]
\centering
\unitlength=0.8mm
\begin{picture}(0,40)(0,0)
\put(0,0){
\begin{picture}(0,30)(0,10)
\thicklines
\put(-10,20){\circle*{3}}
\put(10,20){\circle*{3}}
\put(-10,40){\circle*{3}}
\put(10,40){\circle*{3}}
\qbezier(-10,20)(-10,20)(-10,40)
\qbezier(10,20)(10,20)(10,40)
\qbezier(-10,20)(-10,20)(10,20)
\qbezier(-10,40)(-10,40)(10,40)
\put(0,10){\makebox(0,0)[cc]{\small $C_{4}$}}
\end{picture}}
\put(-60,10){
\begin{picture}(0,40)(0,10)
\thicklines
\put(-10,20){\circle*{3}}
\put(10,20){\circle*{3}}
\qbezier(-10,20)(-5,30)(0,20)
\qbezier(-10,20)(-5,10)(-1,18)
\qbezier(0,20)(5,10)(9,18)
\qbezier(1,22)(5,30)(10,20)
\put(0,0){\makebox(0,0)[cc]{\small $\infty_{+}$}}
\end{picture}}
\put(60,0){
\begin{picture}(0,60)(0,10)
\thicklines
\put(-10,20){\circle*{3}}
\put(10,20){\circle*{3}}
\qbezier(-10,20)(-5,23)(0,20)
\qbezier(-10,20)(-5,17)(-1,19)
\qbezier(0,20)(5,17)(10,20)
\qbezier(1,21)(5,23)(10,20)
\put(-10,40){\circle*{3}}
\put(10,40){\circle*{3}}
\qbezier(-10,40)(-5,43)(0,40)
\qbezier(-10,40)(-5,37)(-1,39)
\qbezier(0,40)(5,37)(10,40)
\qbezier(1,41)(5,43)(10,40)
\qbezier(-10,20)(-13,25)(-10,30)
\qbezier(-10,20)(-7,25)(-9,29)
\qbezier(-10,30)(-7,35)(-10,40)
\qbezier(-11,31)(-13,35)(-10,40)
\qbezier(10,20)(7,25)(10,30)
\qbezier(10,20)(13,25)(11,29)
\qbezier(10,30)(13,35)(10,40)
\qbezier(9,31)(7,35)(10,40)
\put(0,10){\makebox(0,0)[cc]{\small $C_{4} (\infty_{+})$}}
\end{picture}}
\end{picture}
\caption{Diagram $\infty_{+}$, graph $C_{n}$ and disgrsm $C_{n}(\infty_{+})$ for $n=4$.} \label{fig5}
\end{figure}
\noindent
By Theorem \ref{tcn}, the Yamada polynomial of the diagram $C_{n}(\infty_{+})$ presented in Figure~\ref{fig4} is as follows:
$$
R[C_{n}(\infty_{+})]=(-A^{-2}\sigma)^{n}+\sigma(A^{-2}+1)^{n}.
$$

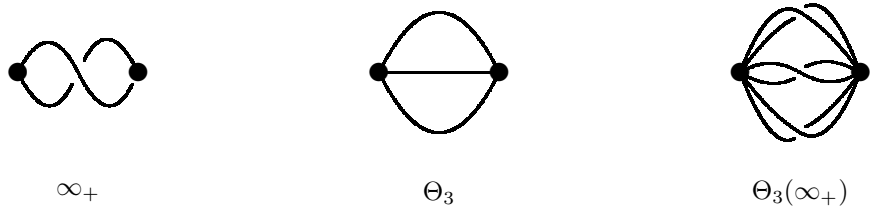
\begin{figure}[!ht]
\centering
\unitlength=0.8mm
\begin{picture}(0,40)
\put(0,10){
\begin{picture}(0,40)(0,10)
\thicklines
\put(-10,20){\circle*{3}}
\put(10,20){\circle*{3}}
\qbezier(-10,20)(0,20)(10,20)
\qbezier(-10,20)(0,00)(10,20)
\qbezier(-10,20)(0,40)(10,20)
\put(0,0){\makebox(0,0)[cc]{\small $\Theta_{3}$}}
\end{picture}}
\put(-60,10){
\begin{picture}(0,40)(0,10)
\thicklines
\put(-10,20){\circle*{3}}
\put(10,20){\circle*{3}}
\qbezier(-10,20)(-5,30)(0,20)
\qbezier(-10,20)(-5,10)(-1,18)
\qbezier(0,20)(5,10)(9,18)
\qbezier(1,22)(5,30)(10,20)
\put(0,0){\makebox(0,0)[cc]{\small $\infty_{+}$}}
\end{picture}}
\put(60,0){
\begin{picture}(0,60)(0,0)
\thicklines
\put(-10,20){\circle*{3}}
\put(10,20){\circle*{3}}
\qbezier(-10,20)(-5,23)(0,20)
\qbezier(-10,20)(-5,17)(-1,19)
\qbezier(0,20)(5,17)(10,20)
\qbezier(1,21)(5,23)(10,20)
\qbezier(-10,20)(-5,33)(0,30)
\qbezier(-10,20)(-5,27)(-1,29)
\qbezier(0,30)(5,27)(10,20)
\qbezier(1,31)(5,33)(10,20)
\qbezier(-10,20)(-5,13)(0,10)
\qbezier(-10,20)(-5,7)(-1,9)
\qbezier(0,10)(5,7)(10,20)
\qbezier(1,11)(5,13)(10,20)
\put(0,0){\makebox(0,0)[cc]{\small $\Theta_{3} (\infty_{+})$}}
\end{picture}}
\end{picture}
\caption{Diagram $\infty_{+}$, graph $\Theta_{n}$ and diagran $\Theta_{n} (\infty_{+})$ for $n=3$.} \label{fig6}
\end{figure}
The Yamada polynomial of the diagram $\Theta_{s}(\infty_{+})$ presented in Figure~\ref{fig5} is as follows:
$$
R[\Theta_{n}(\infty_{+})]=\frac{(-\sigma)^s+\sigma\big((\sigma+1)A^{-2}+1\big)^s}{1+\sigma}.
$$

Similarly, the Yamada polynomial of $B_{q}(\infty_{+})$ is as follows:
$$
R[B_{q}(\infty_{+})]=(-1)^{q-1} \sigma^q.
$$
\end{example}

\section{Zeros of Yamada polynomial}

In this section we will discuss zeros of the Yamada polynomial.  To study zeros of families of polynomials we will base on the following  results by S.~Beraha, J.~Kahane, N.J.~Weiss \cite{BJ} and by A.D.~Sokal \cite{SO}.

\smallskip

\begin{lemma}\label{tbj} \cite{BJ}
If $\{f_n(x)\}$ is a family of polynomials such that
$$
f_n(x) = \alpha_1(x) \lambda_1(x)^n + \alpha_2(x) \lambda_2(x)^n +\ldots + \alpha_l(x) \lambda_l(x)^n,
$$
where the $\alpha_i(x)$ and $\lambda_i(x)$ are fixed non-zero polynomials, such that no pair $i\neq j$ has
$\lambda_i(x) \equiv \omega \lambda_j(x)$ for some complex number $\omega$ of unit modulus. Then $z$ is a limit of zeros of  $\{f_n(x)\}$ if and only if
\begin{itemize}
\item[(1)] two or more of the $\lambda_i(z)$ are of equal modulus, and strictly greater in modulus than the others; or
\item[(2)] for some j, the modulus of $\lambda_j(z)$ is strictly greater than those of the others, and $\alpha_i(z)=0$.
\end{itemize}
\end{lemma}

\smallskip

\begin{lemma} \cite{SO} \label{lso}
Let $F_{1}$, $F_{2}$, and $G$ be analytic functions on a disc $\{ z \in \mathbb C : |z|<R \}$ such that $|G(0)|\leq 1$ and G not constant. Then, for each $\epsilon >0$ there exists $s_{0}<\infty$ such that for all integers $s\geq s_{0}$ the equation
$$
|1+F_{1}(z)G(z)^s|=|1+F_{2}(z)G(z)^s|
$$
has a solution in the disc $|z|<\epsilon$.
\end{lemma}

\smallskip

The study of zeros of Yamada polynomial was started in~\cite{paper1}.

\smallskip

\begin{theorem}\label{tmain3} \cite{paper1}
Zeros of the Yamada polynomial of spatial graphs are dense in the following region:
$$
\Omega =
\big\{ z \in \mathbb C \, : \, |z+1+z^{-1}| \geq \min \{1, |z^{3}+2z^{2}+z+1|, |1 + z^{-1} + 2 z^{-2} + z^{-3} | \} \big\}.
$$
\end{theorem}

\smallskip

So, a natural question arises: whether zeros of Yamada polynomial are dense in the whole complex plane or not? The affirmative answer can be given now.

\smallskip

\begin{theorem}\label{tm4}
Zeros of the Yamada polynomial of spatial graphs are dense in the complex plane.
\end{theorem}

\smallskip

The proof of Theorem~\ref{tm4} is constructive: we will describe explicitly the family of spatial graphs with required property.  The construction will be give in few steps. Firstly, let us denote by $\infty_+^k$ (and $\infty_{-}^{k}$, respectively) diagrams with two vertices and $k$ positive (respectively, negative) double points as shown in Figure~\ref{fig7}.
\begin{figure}[!ht]
\centering
\unitlength=0.7mm
\begin{picture}(0,30)
\put(-60,0){
\begin{picture}(0,30)(-5,0)
\thicklines
\put(-40,10){\circle*{3}}
\put(30,10){\circle*{3}}
\qbezier(-40,10)(-35,20)(-30,10)
\qbezier(-40,10)(-35,0)(-31,8)
\qbezier(-30,10)(-25,0)(-21,8)
\qbezier(-29,12)(-25,20)(-20,10)
\qbezier(-20,10)(-17,5)(-15,5)
\qbezier(-19,12)(-17,15)(-15,15)
\qbezier(30,10)(25,20)(21,12)
\qbezier(30,10)(25,0)(20,10)
\qbezier(15,15)(17,15)(20,10)
\qbezier(15,5)(17,5)(19,8)
\put(-30,20){\makebox(0,0)[cc]{\small $1$}}
\put(-20,20){\makebox(0,0)[cc]{\small $2$}}
\put(20,20){\makebox(0,0)[cc]{\small $k$}}
\put(0,10){\makebox(0,0)[cc]{$\cdots$}}
\end{picture}}
\put(60,0){
\begin{picture}(0,30)(-5,0)
\thicklines
\put(-40,10){\circle*{3}}
\put(30,10){\circle*{3}}
\qbezier(-40,10)(-35,0)(-30,10)
\qbezier(-40,10)(-35,20)(-31,12)
\qbezier(-30,10)(-25,20)(-21,12)
\qbezier(-29,8)(-25,0)(-20,10)
\qbezier(-20,10)(-17,15)(-15,15)
\qbezier(-19,8)(-17,5)(-15,5)
\qbezier(30,10)(25,20)(20,10)
\qbezier(30,10)(25,0)(21,8)
\qbezier(15,5)(17,5)(20,10)
\qbezier(15,15)(17,15)(19,12)
\put(-30,20){\makebox(0,0)[cc]{\small $1$}}
\put(-20,20){\makebox(0,0)[cc]{\small $2$}}
\put(20,20){\makebox(0,0)[cc]{\small $k$}}
\put(0,10){\makebox(0,0)[cc]{$\cdots$}} 
\end{picture}
}
\end{picture}
\caption{The diagram $\infty_{+}^{k}$ (on the left) and $\infty_{-}^{k}$ (on the right).} \label{fig7}
\end{figure}
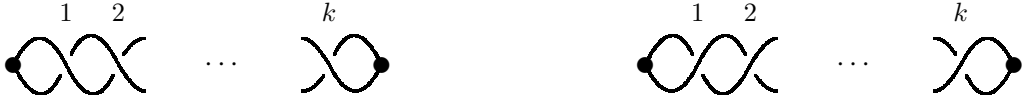

For these diagrams we easy get
\begin{lemma} The following properties hold:
\begin{equation}\label{ek}
R[\infty_+^k]=A^{-2k}\sigma,
\end{equation}
and
\begin{equation}\label{ek'}
R[(\infty_+^k)']=(-1)^{k-1}A^{-k}\sigma (A + A^{-1}) -A^{-2k}\sigma.
\end{equation}
\end{lemma}

\begin{proof}
Indeed, let us denote by $Q_{k}$ the diagram with one $2$-valence vertex and $k$ double points as presented in Figure~\ref{fig8} on left.
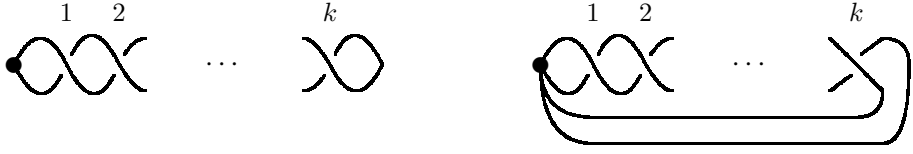
\begin{figure}[!ht]
\centering
\unitlength=0.7mm
\begin{picture}(0,35)(0,-10)
\put(-50,0){\begin{picture}(0,25)
\thicklines
\put(-40,10){\circle*{3}}
\qbezier(-40,10)(-35,20)(-30,10)
\qbezier(-40,10)(-35,0)(-31,8)
\qbezier(-30,10)(-25,0)(-21,8)
\qbezier(-29,12)(-25,20)(-20,10)
\qbezier(-20,10)(-17,5)(-15,5)
\qbezier(-19,12)(-17,15)(-15,15)
\qbezier(30,10)(25,20)(21,12)
\qbezier(30,10)(25,0)(20,10)
\qbezier(15,15)(17,15)(20,10)
\qbezier(15,5)(17,5)(19,8)
\put(-30,20){\makebox(0,0)[cc]{\small $1$}}
\put(-20,20){\makebox(0,0)[cc]{\small $2$}}
\put(20,20){\makebox(0,0)[cc]{\small $k$}}
\put(0,10){\makebox(0,0)[cc]{$\cdots$}}
\end{picture}}
\put(50,0){\begin{picture}(0,25)
\thicklines
\put(-40,10){\circle*{3}}
\qbezier(-40,10)(-35,20)(-30,10)
\qbezier(-40,10)(-35,0)(-31,8)
\qbezier(-30,10)(-25,0)(-21,8)
\qbezier(-29,12)(-25,20)(-20,10)
\qbezier(-20,10)(-17,5)(-15,5)
\qbezier(-19,12)(-17,15)(-15,15)
\qbezier(25,15)(25,15)(21,12)
\qbezier(25,5)(25,5)(20,10)
\qbezier(15,15)(15,15)(20,10)
\qbezier(15,5)(15,5)(19,8)
\put(-30,20){\makebox(0,0)[cc]{\small $1$}}
\put(-20,20){\makebox(0,0)[cc]{\small $2$}}
\put(20,20){\makebox(0,0)[cc]{\small $k$}}
\put(0,10){\makebox(0,0)[cc]{$\cdots$}}
\qbezier(-40,10)(-40,0)(-30,0)
\qbezier(-30,0)(-30,0)(20,0)
\qbezier(20,0)(25,0)(25,5)
\qbezier(-40,10)(-40,-5)(-30,-5)
\qbezier(-30,-5)(-30,-5)(25,-5)
\qbezier(25,-5)(30,-5)(30,10)
\qbezier(30,10)(30,15)(25,15)
\end{picture}}
\end{picture}
\caption{The diagrams $Q_{k}$ and $W_{k}$.} \label{fig8}
\end{figure}

To prove Eq.(\ref{ek}) we observe that  by items (2) and (3) of Proposition~\ref{pr} and item (1) of Proposition~\ref{pv} we get the following relations:
$$
R[Q_{k}] = A R [Q_{k-1}] + A^{-1} \sigma R [Q_{k-1}] - \sigma R[Q_{k-1}] = (A + A^{-1} \sigma - \sigma) R[Q_{k-1}]  = A^{-2} R[Q_{k-1}],
$$
hence $R[Q_{k}]= A^{-2k} \sigma$. By item (3) of Proposition~\ref{pv} we have $R[\infty_{+}^{0}] = R[Q_{0}] = \sigma$.  The induction assumption $R[\infty_{+}^{k-1}] = R[Q_{k-1}]$ implies Eq.(\ref{ek}). Indeed,
$$
R[\infty_{+}^{k}] = A R[\infty_{+}^{(k-1)}] + A^{-1} \sigma R[Q_{k-1}] - \sigma R[\infty_{+}^{{k-1}}] = A^{-2} R[\infty_{+}^{(k-1)}] = R[Q_{k}] = A^{-2k} \sigma.
$$

To demonstrate Eq.(\ref{ek'}) we will also use an induction on $k$. Let  us use  the shorter notation $W_{k} = (\infty_{+}^{k})'$, where $W_{k}$ is the diagram with one 4-valence vertex and $k$ double-points as pictured in Figure~\ref{fig8} on right side. Firstly we verify that Eq.(\ref{ek'}) holds for $k=1$. Indeed, by Example~\ref{E4}, $R[W_{1}] = R[(\infty_+^1)']=\sigma$ and the right side of Eq.(\ref{ek'}) is also equal to $\sigma$. It is easy to see that diagram $W_{k}$ satisfies the skein relation with diagrams obtained from $W_{k}$ by spins $+1$, $-1$, and $0$ at the double point numerated by $k$. For a reader convenience we present these diagrams for the case $k=3$ in Figure~\ref{fig9}. It is easy to see that in general we get $S_{+}(W_{k}) = W_{k-1}$, $S_{-} (W_{k})$ is a dot-product of $Q_{k}$ and the loop $B_{1}$, and $S_{0} (W_{k})$ is a 2-vertex diagram with $(k-1)$ double points which we denote by $X_{k-1}$.
\begin{figure}[!ht]
\centering
\unitlength=0.7mm
\begin{picture}(40,35)(0,-10)
\put(-50,0){\begin{picture}(0,25)
\thicklines
\put(-40,10){\circle*{3}}
\qbezier(-40,10)(-35,20)(-30,10)
\qbezier(-40,10)(-35,0)(-31,8)
\qbezier(-30,10)(-25,0)(-21,8)
\qbezier(-29,12)(-25,20)(-20,10)
\qbezier(-20,10)(-17,5)(-15,5)
\qbezier(-19,12)(-17,15)(-15,15)
\qbezier(-5,15)(-5,15)(-9,12)
\qbezier(-5,5)(-5,5)(-10,10)
\qbezier(-15,15)(-15,15)(-10,10)
\qbezier(-15,5)(-15,5)(-11,8)
\put(-30,20){\makebox(0,0)[cc]{\small $1$}}
\put(-20,20){\makebox(0,0)[cc]{\small $2$}}
\put(-10,20){\makebox(0,0)[cc]{\small $3$}}
\qbezier(-40,10)(-40,0)(-30,0)
\qbezier(-30,0)(-30,0)(-10,0)
\qbezier(-10,0)(-5,0)(-5,5)
\qbezier(-40,10)(-40,-5)(-30,-5)
\qbezier(-30,-5)(-30,-5)(-5,-5)
\qbezier(-5,-5)(0,-5)(0,10)
\qbezier(0,10)(0,15)(-5,15)
\end{picture}}
\put(10,0){\begin{picture}(0,25)
\thicklines
\put(-40,10){\circle*{3}}
\qbezier(-40,10)(-35,20)(-30,10)
\qbezier(-40,10)(-35,0)(-31,8)
\qbezier(-30,10)(-25,0)(-21,8)
\qbezier(-29,12)(-25,20)(-20,10)
\qbezier(-20,10)(-17,5)(-15,5)
\qbezier(-19,12)(-17,15)(-15,15)
\qbezier(-15,15)(-10,10)(-5,15)
\qbezier(-15,5)(-10,10)(-5,5)
\put(-30,20){\makebox(0,0)[cc]{\small $1$}}
\put(-20,20){\makebox(0,0)[cc]{\small $2$}}
\qbezier(-40,10)(-40,0)(-30,0)
\qbezier(-30,0)(-30,0)(-10,0)
\qbezier(-10,0)(-5,0)(-5,5)
\qbezier(-40,10)(-40,-5)(-30,-5)
\qbezier(-30,-5)(-30,-5)(-5,-5)
\qbezier(-5,-5)(0,-5)(0,10)
\qbezier(0,10)(0,15)(-5,15)
\end{picture}}
\put(70,0){\begin{picture}(0,25)
\thicklines
\put(-40,10){\circle*{3}}
\qbezier(-40,10)(-35,20)(-30,10)
\qbezier(-40,10)(-35,0)(-31,8)
\qbezier(-30,10)(-25,0)(-21,8)
\qbezier(-29,12)(-25,20)(-20,10)
\qbezier(-20,10)(-17,5)(-15,5)
\qbezier(-19,12)(-17,15)(-15,15)
\qbezier(-15,15)(-10,10)(-15,5)
\qbezier(-5,15)(-10,10)(-5,5)
\put(-30,20){\makebox(0,0)[cc]{\small $1$}}
\put(-20,20){\makebox(0,0)[cc]{\small $2$}}
\qbezier(-40,10)(-40,0)(-30,0)
\qbezier(-30,0)(-30,0)(-10,0)
\qbezier(-10,0)(-5,0)(-5,5)
\qbezier(-40,10)(-40,-5)(-30,-5)
\qbezier(-30,-5)(-30,-5)(-5,-5)
\qbezier(-5,-5)(0,-5)(0,10)
\qbezier(0,10)(0,15)(-5,15)
\end{picture}}
\put(130,0){\begin{picture}(0,25)
\thicklines
\put(-40,10){\circle*{3}}
\qbezier(-40,10)(-35,20)(-30,10)
\qbezier(-40,10)(-35,0)(-31,8)
\qbezier(-30,10)(-25,0)(-21,8)
\qbezier(-29,12)(-25,20)(-20,10)
\qbezier(-20,10)(-17,5)(-15,5)
\qbezier(-19,12)(-17,15)(-15,15)
\qbezier(-15,15)(-10,10)(-5,5)
\qbezier(-15,5)(-10,10)(-5,15)
\put(-10,10){\circle*{3}}
\put(-30,20){\makebox(0,0)[cc]{\small $1$}}
\put(-20,20){\makebox(0,0)[cc]{\small $2$}}
\qbezier(-40,10)(-40,0)(-30,0)
\qbezier(-30,0)(-30,0)(-10,0)
\qbezier(-10,0)(-5,0)(-5,5)
\qbezier(-40,10)(-40,-5)(-30,-5)
\qbezier(-30,-5)(-30,-5)(-5,-5)
\qbezier(-5,-5)(0,-5)(0,10)
\qbezier(0,10)(0,15)(-5,15)
\end{picture}}
\end{picture}
\caption{The diagrams $W_{3}$, $S_{+}(W_{3})=W_{2}$, $S_{-} (W_{3}) = B_{1} \cdot Q_{2}$ and $S_{0} (W_{3}) = X_{2}$.} \label{fig9}
\end{figure}
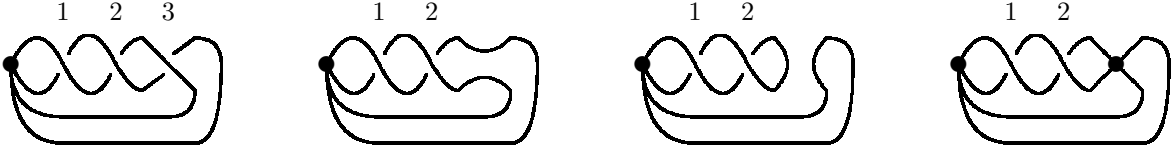

By writing this skein relation (see item (1) of Proposition~\ref{pv}) we get
$$
R[W_{k}] \, = \, A \, R[W_{k-1}] \, + \, A^{-1} R[B_{1} \cdot Q_{k-1}] \, + \, R[X_{k-1}].
$$

Using item (2) of Proposition~\ref{pr} and twice item (2) of Proposition~\ref{pv} we get
\begin{eqnarray*}
R[W_{k}] & = & A R [W_{k-1}] - A^{-1} \sigma R[Q_{k-1}] + (-\sigma) R[W_{k-1}] + R[W_{k-1}] + R[Q_{k-1}] \\
& = & (A - \sigma + 1) R[W_{k-1}] + (1 - A^{-1} \sigma) R[Q_{k}] \\
& = & - A^{-1} R[W_{k-1}] + (-A^{-1} - A^{-2}) A^{-2(k-1)} \sigma.
 \end{eqnarray*}

Assume that Eq.(\ref{ek'}) hold for $k -1$. Then we get
\begin{eqnarray*}
R [W_{k}] & = &  -A^{-1} \left( (-1)^{k-2} A^{-(k-1)} \sigma (A + A^{-1}) - A^{-2(k-1)} \sigma \right) + (-A^{-1} - A^{-2}) A^{-2(k-1)} \sigma \\
& = & (-1)^{k-1} A^{-k} \sigma (A + A^{-1}) + A^{-2k+1} \sigma - A^{-2k+1} \sigma - A^{-2k} \sigma \\
& = & (-1)^{k-1} A^{-k} \sigma (A + A^{-1}) - A^{-2k} \sigma.
\end{eqnarray*}
Hence Eq.(\ref{ek'}) holds for $k$ and the lemma is proven.
\end{proof}

\smallskip

\begin{proof}[Proof of Theorem~\ref{tm4}]
To prove Theorem~\ref{tm4} we will calculate the Yamada polynomial of diagram $C_{n}(\Theta_{s}(\infty_{+}^k))$ which is obtained when we replace every edge of the graph $C_{n}(\Theta_{s})$, having $n$ vertices and $(ns)$ edges as presented in Figure~\ref{fig10} on left, by diagram $\infty_{+}^{k}$ as presented in Figure~\ref{fig10} on right.
\begin{figure}[!ht]
\centering
\unitlength=0.8mm
\begin{picture}(0,40)(0,0)
\put(-60,0){
\begin{picture}(0,30)(0,10)
\thicklines
\put(-10,20){\circle*{3}}
\put(10,20){\circle*{3}}
\put(-10,40){\circle*{3}}
\put(10,40){\circle*{3}}
\qbezier(-10,20)(-10,30)(-10,40)
\qbezier(-10,20)(-14,30)(-10,40)
\qbezier(-10,20)(-6,30)(-10,40)
\qbezier(10,20)(10,20)(10,40)
\qbezier(10,20)(6,30)(10,40)
\qbezier(10,20)(14,30)(10,40)
\qbezier(-10,20)(0,20)(10,20)
\qbezier(-10,20)(0,16)(10,20)
\qbezier(-10,20)(0,24)(10,20)
\qbezier(-10,40)(0,40)(10,40)
\qbezier(-10,40)(0,36)(10,40)
\qbezier(-10,40)(0,44)(10,40)
\put(0,10){\makebox(0,0)[cc]{\small $C_{4} (\Theta_{3})$}}
\end{picture}}
\put(30,0){
\begin{picture}(0,30)(-5,-10)
\thicklines
\put(-40,10){\circle*{3}}
\put(30,10){\circle*{3}}
\qbezier(-40,10)(-35,20)(-30,10)
\qbezier(-40,10)(-35,0)(-31,8)
\qbezier(-30,10)(-25,0)(-21,8)
\qbezier(-29,12)(-25,20)(-20,10)
\qbezier(-20,10)(-17,5)(-15,5)
\qbezier(-19,12)(-17,15)(-15,15)
\qbezier(30,10)(25,20)(21,12)
\qbezier(30,10)(25,0)(20,10)
\qbezier(15,15)(17,15)(20,10)
\qbezier(15,5)(17,5)(19,8)
\put(-30,20){\makebox(0,0)[cc]{\small $1$}}
\put(-20,20){\makebox(0,0)[cc]{\small $2$}}
\put(20,20){\makebox(0,0)[cc]{\small $k$}}
\put(0,10){\makebox(0,0)[cc]{$\cdots$}}
\put(0,-10){\makebox(0,0)[cc]{\small $\infty_{+}^{k}$}}
\end{picture}}
\end{picture}
\caption{Graph $C_{4}(\Theta_{3})$ and diagram $\infty_{+}^{k}$.} \label{fig10}
\end{figure}
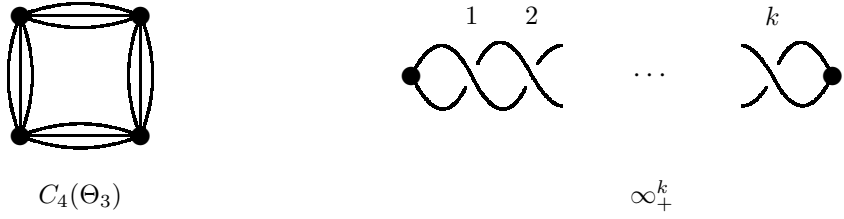

The Yamada polynomial of such a graph can be calculated by the formula given in Corollary~\ref{ccn}:
$$
R[C_{n}(\Theta_{s}(\infty_{+}^k))]= \Big( -R[\Theta_{s}(\infty_{+}^k)] \Big)^n + \sigma \Big( \frac{R[\Theta_{s}(\infty_{+}^k)]+R[(\Theta_{s}(\infty_{+}^k))']}{\sigma} \Big)^n.
$$

By item (1) of Lemma~\ref{tbj}, each complex number $A$ satisfying the equation
\begin{equation}\label{e31}
|-R[\Theta_{s}(\infty_{+}^k)]| = \Big| \frac{R[\Theta_{s}(\infty_{+}^k)] + R[(\Theta_{s}(\infty_{+}^k))']}{\sigma} \Big|,
\end{equation}
is a limit of zeros of Yamada polynomials  $\{ R[C_{n} (\Theta_{s}(\infty_{+}^k))] \}_{n=1,s=1,k=1}^{\infty, \infty, \infty}$.

To make shorter notations in Eq.(\ref{ek'}) we denote
$$
M_{k} = (-1)^{k-1} A^{-k} (A+ A^{-1}).
$$
Then Eq.(\ref{ek'}) can be rewritten as
$$
R[( \infty_{+}^{k})'] = M_{k} \sigma - A^{-2k} \sigma = \sigma (M_{k} - A^{-2k}).
$$
Combining Corollary~\ref{ccn}, Eq.(\ref{ek}) and Eq.(\ref{ek'}) we have
\begin{equation}\label{e32}
\begin{split}
R[\Theta_{s}(\infty_{+}^k)] = & \frac{(-1)^{s}}{1+\sigma} R[(\infty_{+}^k)']^{s} + \frac{\sigma}{1+\sigma} \left( \frac{\sigma+1}{\sigma}R[(\infty_{+}^k)]+\frac{1}{\sigma}R[(\infty_{+}^k)'] \right)^{s} \\
= & \frac{(-1)^{s}}{1+\sigma} \left( M_{k} \sigma  - A^{-2k}\sigma \right)^{s}  + \frac{\sigma}{1+\sigma} \left( \frac{\sigma+1}{\sigma} A^{-2k}\sigma  + \frac{M_{k} \sigma - A^{-2k} \sigma}{\sigma} \right)^{s}\\
= & \frac{(-1)^{s} \sigma^{s}}{1+\sigma} \left( M_{k} - A^{-2k} \right)^{s} + \frac{\sigma}{1+\sigma} \left( M_{k} + A^{-2k}\sigma \right)^{s},\\
\end{split}
\end{equation}
and since $(\Theta_{s})' = B_{s}$ we get
\begin{equation}\label{e33}
R[(\Theta_{s}(\infty_{+}^k))'] =  (-1)^{s-1} \left(R[(\infty_{+}^k)'] \right)^s  = (-1)^{s-1} \sigma^{s} \left( M_{k} - A^{-2k} \right)^{s}.
\end{equation}

Then Eq.(\ref{e31}) is equivalent to
\begin{equation*}
\begin{gathered}
\left| (-1)^{s} \frac{\sigma^{s}}{1 + \sigma} \left( M_{k} - A^{-2k} \right)^{s} +  \frac{\sigma}{1+\sigma} \left( M_{k} + A^{-2k}\sigma \right)^{s} \right| \\
= \left| \frac{1}{\sigma} \left[ (-1)^{s} \frac{\sigma^{s}}{1+\sigma}  \left( M_{k} - A^{-2k} \right)^{s} +  \frac{\sigma}{1+\sigma}  \left( M_{k} + A^{-2k} \sigma \right)^{s} + (-1)^{s-1} \sigma^{s}  \left( M_{k} - A^{-2k} \right)^{s}  \right] \right|,
\end{gathered}
\end{equation*}
hence
\begin{equation*}
\begin{gathered}
\left| (-1)^{s} \sigma^{s} \left( M_{k} - A^{-2k} \right)^{s} +  \sigma \left( M_{k }+ A^{-2k}\sigma \right)^{s} \right|
= \left| (-1)^{s-1} \sigma^{s}   \left( M_{k} - A^{-2k} \right)^{s} +  \left( M_{k} + A^{-2k} \sigma \right)^{s} \right|
\end{gathered}
\end{equation*}
and then
\begin{equation*}
\left| (-1)^{s} + \frac{\sigma \left( M_{k} + A^{-2k} \sigma \right)^{s}}{\sigma^{s} \left( M_{k} - A^{-2k} \right)^{s}} \right| =
\left| (-1)^{s-1} + \frac{\left( M_{k} + \sigma A^{-2k} \right)^{s}}{\sigma^{s} \left( M_{k} - A^{-2k} \right)^{s}} \right|.
\end{equation*}
Denoting
$$
G(A) = \frac{M_{k} + \sigma A^{-2k}}{\sigma (M_{k} - A^{-2k})}
$$
we get that Eq.(\ref{e31}) is equivalent to
$$
\left| (-1)^{s} + \sigma \left[ G(A) \right]^{s} \right| = \left| (-1)^{s-1} + \left[ G(A) \right]^{s} \right|
$$
and each root of this equation is a limit of zeros of polynomials  $\{ R[C_{n} (\Theta_{s}(\infty_{+}^k))] \}_{n=1,s=1,k=1}^{\infty, \infty, \infty}$. Observe that for $s$ even (respectively, odd) we have
\begin{equation} \label{e34}
\left| 1 + \sigma \left[ G(A) \right]^{s} \right| = \left| 1 - \left[ G(A) \right]^{s} \right| \quad \text{\rm or} \quad \left| 1 - \sigma \left[ G(A) \right]^{s} \right| = \left| 1 + \left[ G(A) \right]^{s} \right|,
\end{equation}
which are both of the form presented in Lemma~\ref{lso}. We will discuss one of them, and the other can be considered analogously. 

Let us fix a complex number $z_{0}$ such that $|z_{0}| \leq 1$,  and denote $a = A-z_{0}$. Recall that
$$
\sigma = A + 1 + A^{-1} = (a + z_{0}) + 1 + (a+z_{0})^{-1}
$$
and
$$
M_{k} = (-1)^{k-1} A^{-k} (A + A^{-1}) = (-1)^{k-1} (a+z_{0})^{-k} ((a+z_{0}) + (a+z_{0})^{-1}).
$$
Then left equation in Eq.(\ref{e34}) is equivalent to
\begin{equation}
\left| 1 + ((a + z_{0}) + 1 + (a+z_{0})^{-1}) \left[ G(a) \right]^{s} \right| = \left| 1 - \left[ G(a) \right]^{s} \right|
\end{equation}
where
$$
G(a) =  \frac{(-1)^{k-1} (a+z_{0})^{-k} \Big( (a+z_{0}) + (a+z_{0})^{-1} \Big) + \Big( (a + z_{0}) + 1 + (a+z_{0})^{-1} \Big)  (a+ z_{0})^{-2k}}{\Big( (a + z_{0}) + 1 + (a+z_{0})^{-1}\Big) \Big( (-1)^{k-1} (a+z_{0})^{-k} ((a+z_{0}) + (a+z_{0})^{-1}) - (a + z_{0})^{-2k}  \Big) }.
$$
and we have
\begin{equation} \label{e100}
|G(0)|  =  \Big| \frac{(-1)^{k-1} z_{0}^{-k}(z_{0} + z_{0}^{-1})+ (z_{0} + 1 + z_{0}^{-1} ) z_{0}^{-2k}}{(z_{0} + 1 + z_{0}^{-1}) (-1)^{k-1} z_{0}^{-k} (z_{0} + z_{0}^{-1}) + (z_{0} + 1 + z_{0}^{-1} ) z_{0}^{-2k}} \Big| .
\end{equation}
It will follows from Lemma~\ref{lso}, that for any $\varepsilon > 0$ there exists such $k_{0}$ that for any $k \geq k_{0}$ there exists a root of the equation Eq.(\ref{e34}) in the disk $|a| < \varepsilon$, and therefore, there exists a root $A$ of the Yamada polynomial  $R[C_{n}(\Theta_{s}(\infty_{+}^k))]$ in the $\varepsilon$-neighborhood of $z_{0}$ if $|G(0)| \leq 1$.

Now we consider $G(0)$ as a function
\begin{equation} \label{eq101}
f(z)  = \frac{(-1)^{k-1} z^{-k}(z + z^{-1})+ (z + 1 + z^{-1} ) z^{-2k}}{(z + 1 + z^{-1}) (-1)^{k-1} z^{-k} (z + z^{-1}) + (z + 1 + z^{-1} ) z^{-2k}} .
\end{equation}
calculates in the complex number $z_{0}$.  We will show that there exists a complex number $z_{1}$ in the $\varepsilon$-neighborhood of $z_{0}$ such that $|f(z_{1})| = 1$.

It is easy to see that the equality $| f(z) |=1$ is equivalent to the equality
$$
\Big| 1 + (-1)^{k-1} \frac{z + z^{-1}}{z + 1 + z^{-1}} z^{k}  \Big| = \Big| 1 + (-1)^{k} (z+z^{-1}) z^{k} \Big| .
$$
Since we are looking roots in a neighborhood of $z_{0}$, denote $\zeta = z - z_{0}$ and consider function
\begin{equation} \label{eq102}
\tilde f(\zeta)  =  \frac{(-1)^{k-1} z^{-k}(z + z^{-1})+ (z + 1 + z^{-1} ) z^{-2k}}{(z + 1 + z^{-1}) (-1)^{k-1} z^{-k} (z + z^{-1}) + (z + 1 + z^{-1} ) z^{-2k}},
\end{equation}
where $z = z_{0} + \zeta$. The equality $| \tilde f(\zeta) | = 1$ is equivalent to the equality
\begin{eqnarray} \label{e103}
\begin{gathered}
\Big| 1 + (-1)^{k-1} \frac{(z_{0} + \zeta) + (z_{0} + \zeta)^{-1}}{(z_{0} + \zeta) + 1 + (z_{0} + \zeta)^{-1}} (z_{0} + \zeta)^{k}  \Big| \qquad \qquad  \\  \qquad \qquad  = \Big| 1 + (-1)^{k} ((z_{0} + \zeta)+(z_{0} + \zeta)^{-1}) (z_{0} + \zeta)^{k} \Big| .
\end{gathered}
\end{eqnarray}
Denote  $g(\zeta) = (z_{0} + \zeta)$. Since $|z_{0}| \leq1$, we have $|g(0)| \leq 1$. Therefore we can apply Lemma~\ref{lso} to Eq.(\ref {e103}). It gives us that for any $\varepsilon>0$ there exists $k_{1}$ such that for any integers $k \geq k_{1}$ there exists a root $\zeta_{0}$ of the Eq.(\ref{e103}) such that $|\zeta_{0}| < \varepsilon$.  Thus, we obtained that there exists $z_{1} = z_{0} + \zeta_{0}$ such that $|z_{1} - z_{0}| < \varepsilon$ and $|f(z_{1})| = 1$ in Eq.(\ref{e100}).  Hence, if we replace $z_{0}$ by $z_{1}$ in Eq.(\ref{e34}) and Eq.(\ref{e100}), we will get that there exists a number $A$ which is a limit of zeros of Yamada  polynomials  $\{ R[C_{n} (\Theta_{s}(\infty_{+}^k))] \}_{n=1,s=1,k=1}^{\infty, \infty, \infty}$ such that $|z_{1} - A| < \varepsilon$ and $|z_{0} - A| < 2 \varepsilon$. Thus, we proved that roots of Yamada polynomial are dense in the disc $\{ z \in \mathbb C :  |z| \leq 1\}$.

The relation between the Yamada polynomials of a spatial graph $g$ and its mirror image $\widehat{g}$ is presented in Proposition~6 in \cite{YA}:
$$
R[\widehat{g}](A) = R[g](A^{-1}).
$$

Since diagram $C_{n}(\Theta_{s}(\infty_-^k))$ is the mirror image of diagram $C_{n}(\Theta_{s}(\infty_+^k))$, from the previous discussions we see that zeros of Yamada polynomials of $R[C_{n}(\Theta_{s}(\infty_-^k))]_{n=1,s=1,k=1}^{\infty, \infty, \infty}$ are dense in the region $\{z\in \mathbb{C} \, : \,  |z|\geq 1\}$.

Summarising, we obtain that the zeros of Yamada polynomial are dense in the complex plane.
\end{proof}


\bibliographystyle{amsplain}

\end{document}